\documentclass[11pt,a4paper]{amsart}%
\usepackage{amssymb}
\usepackage{amsmath}
\usepackage{amsfonts}
\usepackage{enumitem}
\usepackage{bm}
\usepackage{comment}
\usepackage{pgfplots}
\usepackage{xurl}
\usepackage{float}
\usepackage[T1]{fontenc}
\usepackage[utf8]{inputenc}

\providecommand{\U}[1]{\protect\rule{.1in}{.1in}}
\usetikzlibrary{backgrounds}
\usetikzlibrary{arrows}
\pgfplotsset{compat=1.15}
\let\oldmathbf\mathbf
\renewcommand{\mathbf}[1]{\boldsymbol{\oldmathbf{#1}}}
\newtheorem{theorem}{Theorem}
\newtheorem*{theorem*}{Theorem}

\newtheorem{corollary}[theorem]{Corollary}

\newtheorem{definition}[theorem]{Definition}

\newtheorem{lemma}[theorem]{Lemma}

\newtheorem{remark}[theorem]{Remark}

\newcommand{\mycomment}[1]{}
\allowdisplaybreaks
\definecolor{zzttqq}{rgb}{0.6,0.2,0.}
\allowdisplaybreaks
\begin{document}
\title[Irregularities of distribution and convex bodies]{\textbf{Irregularities of distribution and Fourier transforms of multi-dimensional convex bodies}}
\author[L. Brandolini]{Luca Brandolini}
\address{Dipartimento di Ingegneria Gestionale, dell'Informazione e della Produzione,
Universit\`{a} degli Studi di Bergamo, Viale Marconi 5, 24044 Dalmine BG, Italy}
\email{luca.brandolini@unibg.it}
\author[L. Colzani]{Leonardo Colzani}
\address{Dipartimento di Matematica e Applicazioni, Università di Milano-Bicocca,
Via Cozzi 55, 20125 Milano, Italy}
\email{leonardo.colzani@unimib.it}
\author[G. Travaglini]{Giancarlo Travaglini}
\address{Dipartimento di Matematica e Applicazioni, Universit\`{a} di Milano-Bicocca,
Via Cozzi 55, 20125 Milano, Italy}
\email{giancarlo.travaglini@unimib.it}
\begin{abstract}
W. Schmidt, H. Montgomery, and J. Beck  proved a result on irregularities of distribution with respect to $d$-dimensional balls. In this paper, we extend their result to any $d$-dimensional convex body with a smooth boundary and  finite order of contact. As an intermediate step, we prove a geometric inequality for the Fourier transform of the characteristic function of a convex body.
\end{abstract}
\maketitle
\section{Introduction}
W. Schmidt \cite{schmidt69}, H. Montgomery \cite{montgomery84,montgomery94}, and J. Beck \cite{beck} proved the following result.
\begin{theorem}
\label{SMBT}
Let $d\geqslant2$, and let  $
\mathbb{T}^{d}=\mathbb{R}^{d}/\mathbb{Z}^{d}$ be the $d$-dimensional torus.  
Let $B\subset
\mathbb{T}^{d}$ be a ball of diameter less than $1$. Then there exists a positive constant
$c$ such that, for any choice of $N$\ points $z_{1},\ldots,z_{N}$ in $\mathbb{T}^{d}$, we have%
\begin{equation}
\int_{0}^{1}\int_{\mathbb{T}^{d}}\left\vert \sum_{j=1}^{N}%
\chi_{ r  B+x}\left(  z_{j}\right)  - N r ^{d}\left\vert B\right\vert
\right\vert ^{2}~dxd r \geqslant c  N^{1-1/d}\;, \label{SMB}
\end{equation}
where $\left\vert B\right\vert $ is the volume of the ball $B$, and $\chi_K(t)$ denotes the characteristic (or indicator) function of a set $K$.
\end{theorem}

An earlier result by D. Kendall \cite{Ken} for ellipsoids shows that the lower bound  $N^{1-1/d}$  in (\ref{SMB}) is best possible.

A natural question concerns the necessity of the average over the radii $r$ on the LHS of (\ref{SMB}).
The result in \cite{ParSob} (see also  \cite{BCGT,ChenTrav}) on the lattice points discrepancy shows that an average is necessary when the dimension $d\equiv 1\mod 4$. We do not know whether an average is necessary when $d\not\equiv 1\mod 4$, nevertheless in \cite{BMS} and \cite{BGGM} the authors consider analogs of the above result for single radius spherical cap discrepancy on spheres and compact two-point homogeneous spaces, showing that when $d\not\equiv 1\mod 4$ a suitable choice of the radius suffices.

\medskip

The relevance of Theorem \ref{SMBT}  becomes more evident when comparing  the discrepancy with respect to balls to the discrepancy with respect to cubes in the celebrated result by K. Roth \cite{roth}, as stated here in the equivalent form proved by H. Montgomery \cite{montgomery84,montgomery94}. See also \cite{drmota1}. 
\begin{theorem} \label{Ro}
Let $d\geqslant2$ and consider the cube $Q=\left[  0,1\right)^{d}$. Then, for every choice of
$N$ points $z_{1},\ldots ,z_{N}$  in 
$\mathbb{T}^{d}$, we have%
\begin{equation}
\int_{0}^{1}\int_{\mathbb{T}^{d}}\left\vert \sum_{j=1}^{N}%
\chi_{ r  Q+x}\left(  z_{j}\right)  - N r ^{d}\right\vert ^{2}%
~dxd r \geqslant c  \log^{d-1}\left(  N\right)  \;. \label{R}
\end{equation}
\end{theorem}
H. Davenport \cite{davenport} and K. Roth \cite{roth80} have shown that the lower bound in (\ref{R}) cannot be improved, see also \cite{chen15}. M. Drmota \cite{drmota1} has extended the above result from the cube to a convex polyhedron.

\medskip

Theorem \ref{SMBT} and Theorem \ref{Ro} are fundamental results in the
field termed  \textquotedblleft irregularities of distribution\textquotedblright \ or  \textquotedblleft geometric discrepancy\textquotedblright\  (see \cite{BC}, see also \cite{Bi,Cha,CST,DP,DT,matousek,travaglini}). Together,  the above two theorems show that  the discrepancy -- i.e.,  the comparison between the atomic measure associated with $z_1,  \ldots , z_N$ and the Lebesgue measure -- appears to vary significantly depending on the family of sets on which it is tested. We are interested in a geometric explanation for this different behavior.

\medskip

In this paper, we replace the ball $B$ with a convex body (that is a compact convex set with nonempty interior) $\Omega\subset
\mathbb{T}^{d}$ having a smooth boundary $\partial \Omega$ with finite order of contact at every point. Specifically, at each point of $\partial \Omega$, every one-dimensional tangent line has finite order of contact with $\partial \Omega$.  In particular, $\Omega$ has to be strictly convex.

Our main result is stated in the following theorem, which suggests that the smoothness of  $\partial \Omega$ accounts for the observed difference between the estimate in (\ref{SMB}) and the estimate in (\ref{R}). Incidentally we observe that our result does not require the curvature of $\partial \Omega$ to be positive everywhere. An analog of Theorem \ref{d3} for the planar case has already been proved in \cite{BT22}. See also  \cite{BCGGT,BT20,G2020} for earlier results.

\begin{theorem} \label{d3} Let $\Omega$ be a  convex body
in $\mathbb{T}^{d}$. Assume that $\partial\Omega$ is smooth with
finite order of contact at every point. Then there exists a positive constant $c$ such that for every choice of $N$ points $z_1, \ldots , z_N$
in $\mathbb{T}^{d}$ we have
\[
\int_{\mathbb{T}^{d}}\int_{0}^{1}\left|\sum_{j=1}^{N}\chi_{x+r\Omega}\left(z_{j}\right)-N r^d \left|\Omega\right|\right|^{2}dxdr\geqslant c\,  N^{1-1/d}.
\]
In particular, for every choice of $N$ points $z_1, \ldots , z_N$
in $\mathbb{T}^{d}$ there exists a dilated and translated copy $x+r\Omega$ of $\Omega$ with discrepancy
\[
\left|\sum_{j=1}^{N}\chi_{x+r\Omega}\left(z_{j}\right)-N r^d \left|\Omega\right|\right|\geqslant c\,  N^{1/2-1/2d}.
\]
\end{theorem}

The proof of Theorem \ref{d3} depends on suitable  geometric estimates for the Fourier transform $\widehat{\chi_{\Omega}}$, see Section \ref{due}, and on (a variation of) a lemma of Cassels and Montgomery (see \cite{montgomery94},
see also \cite{BBG} for a more general result).

\section{Geometric estimates for the Fourier transform} \label{due}
The following estimate is classical (see e.g. \cite{G-G-V}, \cite{H} or \cite{Hl}).
\begin{theorem}\label{Herz}
Let $\Omega\subset\mathbb{R}^{d}$ be a convex body with smooth boundary having everywhere positive Gaussian curvature. For every $\xi\in\mathbb{R}^{d}\setminus\left\{ 0\right\}$  let 
$p\left(\xi\right)$ be the unique point on the boundary $\partial\Omega$  with outward unit normal $\xi/\left\vert \xi\right\vert$. Also let $K\left(\xi\right)$ be the Gaussian curvature of $\partial\Omega$  at $p\left(\xi\right)$. Then, as $\left\vert \xi\right\vert \rightarrow+\infty$, the Fourier transform of $\chi_{\Omega}\left(x\right)$ has the asymptotic expansion	
\begin{align*}
& \widehat{\chi_{\Omega}}\left(\xi\right) =\\
 &	\frac{-1}{2\pi i}\left\vert \xi\right\vert ^{-\frac{d+1}{2}}\left[K^{-\frac{1}{2}}\left(\xi\right)e^{-2\pi i\left(\xi\cdot p\left(\xi\right)-\frac{d-1}{8}\right)}-K^{-\frac{1}{2}}\left(-\xi\right)e^{-2\pi i\left(\xi\cdot p\left(-\xi\right)+\frac{d-1}{8}\right)}\right] \\
	 & +O\left(\left\vert \xi\right\vert ^{-\frac{d+3}{2}}\right).
\end{align*}
\end{theorem}

In the directions where the Gaussian curvature vanishes, the order of decay of the Fourier transform changes. 
Nevertheless, in \cite{BNW},  J. Bruna, A. Nagel, and S. Wainger  provided the following elegant uniform geometric estimate of the Fourier transform by replacing the curvature with the size of spherical caps on $\partial\Omega$.

\begin{theorem} \label{BNWtheorem}
Let $d \geqslant 2$ and let $\Omega \subset \mathbb{R}^d$ be a convex body whose boundary $\partial \Omega$ has finite order of contact at every point. Let $d\sigma$ be the measure induced  on $\partial \Omega$ by the Lebesgue measure and let $\eta$ be a smooth function on $\partial\Omega$. Let $\Theta \in \mathbb{R}^d$, $\left\vert \Theta \right\vert =1$, and let $x_1, x_2$ be the two points in $\partial \Omega$ where the tangent planes \ $T_{x_{1}}, T_{x_{2}} $ are orthogonal to $\Theta$. For every  $\rho>0$ and $j=1,2$, we consider the \textquotedblleft spherical cap\textquotedblright 
\[
B_{x_j, \rho^{-1}}=\left\{y\in \partial \Omega : \mathrm{dist}\left(y,T_{x_{j}}\right) < \rho^{-1} \right\} \ ,
\]
and its  surface measure $\sigma\left( B_{x_j, \rho^{-1}} \right)$.  Then
\[
\left\vert \int_{\partial \Omega} \eta(x) e^{-2\pi i \rho \Theta \cdot x} \ d\sigma(x)  \right\vert \leqslant c \Big(  \sigma\left( B_{x_1, \rho^{-1}} \right) + \sigma\left( B_{x_2, \rho^{-1}} \right) \Big) \ , 
\]
where $c $ only depends on $\Omega$ and $\eta$.
\end{theorem}

In fact, this is not exactly the form in which Bruna, Nagel, and Wainger state their result, as they use a cutoff function of suitable small support to isolate the critical points of the oscillatory integral. 
Our statement follows from theirs by a suitable partition of unity.

We remark that in the case $d=2$ the above result holds if $\partial\Omega$ is a smooth convex curve without assuming finite order of contact. See Section 6 in \cite{BNW}. 

\medskip
As an immediate consequence of the result of Bruna, Nagel, and Wainger, one obtains an estimate for the Fourier transform of the characteristic function of $\Omega$, where the spherical caps are replaced with the \textquotedblleft parallel section functions\textquotedblright \ introduced in  the following definition. 
\begin{definition}
Let $\Omega \subset \mathbb{R}^d$ be a convex body. Let $\Theta \in \mathbb{R}^d$,  $|\Theta|=1$. For $t \in \mathbb{R}$ the function 
\[
A_{\Omega , \Theta} (t)= \mathrm{Vol}_{d-1} \left(\Omega \cap \left\{\Theta^{\perp} +t\Theta\right\}  \right)
 \]
is  the parallel section function of $\Omega$ (that is the Radon transform)  in the direction $\Theta$ (here $\Theta^{\perp}=\left\{x\in \mathbb{R}^d: x \cdot \Theta =0 \right\}$).

Up to a translation of $\Omega$, we can assume that the origin is an interior point of  $\Omega$, so that for every direction $\Theta$ the parallel section function $A_{\Omega,\Theta}$ has support in an interval $\left[-a,b\right]$, where $a=a(\Theta)$ and $b=b(\Theta)$ are positive functions. 
\end{definition}

\begin{corollary} \label{BNW1}
Let $d \geqslant 2$ and let $\Omega \subset \mathbb{R}^d$ be a convex body whose boundary $\partial \Omega$ has finite order of contact at every point. Then there exist positive constants $c$ and $\kappa$ such that for every $\rho\geqslant\kappa$, and every direction $\Theta$ we have
\begin{equation*}
\left\vert\widehat{\chi}_{\Omega}(\rho \Theta) \right\vert \leqslant c  \ \rho^{-1} \left(A_{\Omega , \Theta} (-a+\rho^{-1})+A_{\Omega , \Theta} (b-\rho^{-1})\right) \ .
\end{equation*}
\end{corollary}
Observe that when the Gaussian curvature of $\partial\Omega$ is positive, then 
$A_{\Omega , \Theta} (-a+\rho^{-1})$ and $A_{\Omega , \Theta} (b-\rho^{-1})$ are of the order of 
$\rho^{-(d-1)/2}$,  giving the well known estimate for a body with positive curvature
\[
\left\vert\widehat{\chi}_{\Omega}(\rho \Theta) \right\vert \leqslant c \rho^{-(d+1)/2},
\]
that follows from Theorem \ref{Herz}. 
\begin{proof}[Proof of Corollary \ref{BNW1}] Let $x \in \partial \Omega$ and let $n(x)$ be the outward unit normal at $x$. Then the divergence theorem yields
\begin{align*}
\widehat{\chi_{\Omega}}\left(\rho\Theta\right)= & \int_{\Omega}e^{-2\pi i\rho\Theta\cdot x}dx =  -\frac{1}{4\pi^{2}\rho^{2}}\int_{\Omega}\Delta\left(e^{-2\pi i\rho\Theta\cdot x}\right)dx\\
= & -\frac{1}{4\pi^{2}\rho^{2}}\int_{\partial\Omega}\nabla\left(e^{-2\pi i\rho\Theta\cdot x}\right)\cdot n\left(x\right)dx\\
= & \frac{i}{2\pi \rho}\int_{\partial\Omega}e^{-2\pi i\rho\Theta\cdot x} \ \Theta\cdot n\left(x\right)d\sigma\left(x\right)\\
= & \frac{i}{2\pi \rho}\sum_{h=1}^{d}\Theta_{h}\int_{\partial\Omega}e^{-2\pi i\rho\Theta\cdot x}n_{h}\left(x\right)d\sigma\left(x\right).
\end{align*}
Then the corollary follows from Theorem \ref{BNWtheorem} since the surface measure of the spherical cap and the parallel section function are comparable.
\end{proof}

A. Podkorytov (see \cite{Pod91}, see also \cite{BRT}) proved that, when $d=2$, the above corollary also holds for a convex domain $\Omega$, without  smoothness  assumptions on $\partial\Omega$. 
On the other hand, in dimension $d \geqslant 3$ convexity alone is not enough and some smoothness of $\partial \Omega$ seems to be necessary. Indeed, consider for example the double cone 
\[
\Omega = \left\{ (x,y,z) \in \mathbb{R}^3: \sqrt{y^2+z^2}\leqslant 1-|x|\right\}.
\]
Then for every $\rho>1$ the volume of the set 
\[
\Omega \cap \left\{(x,y,z): x \geqslant 1-\rho^{-1}\right\}
\]
 is $\rho^{-3} \pi/3$, whereas the Fourier transform of $\chi_{\Omega}$ at the point $(\rho,0,0)$ is given by
 \begin{equation*}
\widehat{\chi}_{\Omega}(\rho,0,0) = \int_{-1}^1 \pi \left(1-|x|\right)^2 e^{-2\pi i x \rho} dx =
 \frac1{\pi \rho^2} - \frac{\sin\left(2 \pi \rho \right)}{2\pi^2 \rho^3} \ . 
\end{equation*}
Hence \ $\left|\widehat{\chi}_{\Omega}(\rho,0,0)\right|$\  cannot be bounded from above with the volume of the set \ $\Omega \cap \left\{(x,y,z): x \geqslant 1-\rho^{-1}\right\}$.
 
\medskip

The following classical result by H. Brunn is a corollary of the Brunn-Minkowski inequality. See e.g. \cite{Kold}. 
\begin{theorem}\label{Brunn}
Let $d \geqslant 2$, let $\Omega \subset\mathbb{R}^d$ be a convex body, let $\Theta \in \mathbb{R}^d$, $\left\vert \Theta \right\vert =1$, 
and let $A_{\Omega , \Theta} (t)$ be the parallel section function in the direction $\Theta$. Then the function $\left(A_{\Omega , \Theta} (t)\right)^{1/(d-1)}$ is concave on its support.
\end{theorem}

We observe that the  $d$-dimensional Fourier transform of $\chi_{\Omega}$ is the Fourier transform of the parallel section function, namely  
\begin{equation}
\widehat{\chi}_{\Omega}(\rho \Theta) = \widehat{A_{\Omega , \Theta}} \left(\rho\right) \ . \label{Chi=A}
\end{equation}
Hence, when $d=2$ the study of $\widehat{\chi}_\Omega$ reduces to the study of the Fourier transform of a concave (on its support) one-dimensional function, whereas for the case $d\geqslant 3$ we only have a power of a concave function. This seems to be a relevant difference: for instance, observe that  the derivative of a concave function is monotone, whereas this may  not be true for the derivative of a power of a concave function.

\medskip

The next two theorems  provide estimates from below for the decay of the Fourier transform. The first estimate, for bodies with positive curvature, is an easy consequence of Theorem \ref{Herz}.  The second estimate, for general bodies with boundary of finite type, can be seen as a sort of converse of Corollary \ref{BNW1} (hence essentially  a converse of Theorem \ref{BNWtheorem}). It is slightly less precise than Theorem \ref{easy}, nevertheless it suffices to prove Theorem \ref{d3}. 
\begin{theorem}\label{easy}
Let $\Omega \subset \mathbb{R}^d$ be a convex body whose boundary $\partial \Omega$ has positive Gaussian curvature. Then there exists a positive constant $C$ such that for every direction $\Theta$ and $\rho\geqslant1,$
\begin{align*}
\int_\rho^{\rho+1} \left|\widehat{\chi_{\Omega}}\left(s\Theta\right)\right|^{2}ds
 \geqslant {C}{\rho^{-(d+1)}}.
\end{align*}
\end{theorem}

\begin{theorem}\label{thm:da sotto}
Let $\Omega \subset \mathbb{R}^d$ be a convex body whose boundary $\partial \Omega$ has finite order of contact at every point. Assume that the origin is an interior point of $\Omega$ so that for every direction $\Theta$ the parallel section function $A_{\Omega,\Theta}$ has support in an interval $\left[-a,b\right]$, where $a=a(\Theta)$ and $b=b(\Theta)$ are positive functions. Then there exist positive constants $\alpha,\beta,c,C$ such that for every direction $\Theta$ and $\rho>c,$ 
\begin{align}
\int_{\alpha\rho\leqslant s\leqslant\beta\rho} \left|\widehat{\chi_{\Omega}}\left(s\Theta\right)\right|^{2}ds \label{needave} 
 \geqslant \frac{C}{\rho}\left[A_{\Omega,\Theta}\left(-a+\rho^{-1}\right)^{2}+A_{\Omega,\Theta}\left(b-\rho^{-1}\right)^{2}\right]. 
\end{align}
\end{theorem}
Assuming that the origin is an interior point of $\Omega$ is not necessary, but simplifies the notation and does not affect the generality of the theorem since a translation does not change the modulus of the Fourier transform.

A different proof of the above result has been shown to us by A. Podkorytov \cite{PodPersCom}. Our proof is an elaboration of ideas in \cite{BCT} and \cite{BT22}.

Observe that, since  the function $\widehat{\chi_{\Omega}}$ may have infinitely many zeros, we cannot bound $\left\vert\widehat{\chi_{\Omega}}\left(\rho\Theta\right)\right\vert$  from below for every $\rho$. Then an average, such as  \ $\int_{\alpha\rho\leqslant s\leqslant\beta\rho}\left|\widehat{\chi_{\Omega}}\left(s\Theta\right)\right|^{2}ds$,  is necessary. 

\section{Proof of Theorem \ref{easy} and Theorem \ref{thm:da sotto}} \label{tre}
Theorem \ref{easy} is an easy consequence of the asymptotic estimate for the Fourier transform of a convex body with positive Gaussian curvature.
\begin{proof}[Proof of Theorem \ref{easy}]
By Theorem \ref{Herz}
\begin{align*}
& \left|\widehat{\chi_{\Omega}}\left(\xi\right) \right| \\
  = &	\frac{1}{2\pi}\left\vert \xi\right\vert ^{-\frac{d+1}{2}}\left|K^{-\frac{1}{2}}\left(\xi\right)e^{-2\pi i\left(\xi\cdot p\left(\xi\right)-\frac{d-1}{8}\right)}-K^{-\frac{1}{2}}\left(-\xi\right)e^{-2\pi i\left(\xi\cdot p\left(-\xi\right)+\frac{d-1}{8}\right)}\right| \\
	 & +O\left(\left\vert \xi\right\vert ^{-\frac{d+3}{2}}\right) \\
	 = & \frac1{2\pi} \left|\xi\right|^{-\frac{d+1}2}  \left|K^{-\frac{1}{2}}\left(\xi\right) + K^{-\frac{1}{2}}\left(-\xi\right) e^{2\pi i\left(\xi\cdot \left(p(\xi)-p(-\xi)  \right) -\frac{d-1}4\right)}\right|
	  +O\left(\left\vert \xi\right\vert ^{-\frac{d+3}{2}}\right).
\end{align*}
Hence,
\begin{align*}
& \int_\rho^{\rho+1} \left|\widehat{\chi_{\Omega}}\left(s\Theta\right) \right|^2 ds \\
 \geqslant  & c \rho^{-(d+1)} \int_\rho^{\rho+1}    \left|K^{-\frac{1}{2}}\left(\Theta\right) + K^{-\frac{1}{2}}\left(-\Theta\right) e^{2\pi i\left(s\Theta\cdot \left(p(\Theta)-p(-\Theta)  \right) -\frac{d-1}4\right)}\right|^2 ds \\
 & + O\left( \rho^{-(d+2)}  \right) \\
 \end{align*}
 Observe that, being $\Omega$ a convex body, the phase $\sigma = \Theta\cdot \left(p(\Theta)-p(-\Theta)  \right)\neq0$ for every direction $\Theta$.  Hence, we are lead to estimate an integral of the form
\begin{align*}
 & \int_\rho^{\rho+1}    \left|A + B e^{2\pi i\left(s \sigma +\tau \right)}\right|^2 ds \\
  = & A^{2}+B^{2}+2AB\int_{\rho}^{\rho+1}\cos\left(2\pi \sigma s+\tau\right)ds\\
 = & \left(A-B\right)^{2}+2AB\left(1+\int_{\rho}^{\rho+1}\cos\left(2\pi \sigma s+\tau\right)ds\right).
 \end{align*}
 Hence
 \[
 \int_\rho^{\rho+1}  \left|A + B e^{2\pi i\left(\rho \sigma +\tau \right)}\right|^2 ds \geqslant
 2AB\left(1+\int_{\rho}^{\rho+1}\cos\left(2\pi \sigma s+\tau\right)ds\right).
 \]
Finally there exists $c>0$ such that, for every $\rho$,
 \[
 1+ \int_{\rho}^{\rho+1}\cos\left(2\pi \sigma s+\tau\right)ds > c.
 \]
\end{proof}
The proof of Theorem \ref{thm:da sotto} requires more work. We start with the following lemmas.
\begin{lemma}
\label{lem:uniform ball condition}There exists $c>0$ such that for $\rho$ large enough
and for every $\Theta$ we have
\[
A_{\Omega,\Theta}\left(-a+\rho^{-1}\right)\geqslant c\rho^{-\left(d-1\right)/2}
\]
and
\[
A_{\Omega,\Theta}\left(b-\rho^{-1}\right)\geqslant c\rho^{-\left(d-1\right)/2}.
\]
\end{lemma}

\begin{proof}
Let ${n}\left(x\right)$ denote the inward normal at $x\in\partial\Omega$.
By the Tubular Neighborhood Theorem (see e.g.\ Theorem 6.24 in \cite{Lee})
there exists $\varepsilon>0$ such that the map
\[
T\left(x,\alpha\right)=x+\alpha{n}\left(x\right)
\]
defines a diffeomorphism between $\partial\Omega\times\left(-\varepsilon,\varepsilon\right)$
and a neighborhood $U$ of $\partial\Omega$.

Now, let $y=x+\alpha{n}\left(x\right)\in U$, with $x\in\partial\Omega$.
We claim that $d\left(y,\partial\Omega\right)=\alpha$, and that $x$
is the unique point on $\partial\Omega$ such that $d\left(x,y\right)=\alpha$.
To prove this, suppose there exists $z\in\partial\Omega$ such that
$d\left(z,y\right)=\beta\leqslant\alpha$. Then we would have 
\[
y=x+\alpha{n}\left(x\right)=z+\beta{n}\left(z\right)
\]
contradicting the injectivity of $T$. This, in particular, implies
that $x+\frac{1}{2}\varepsilon{n}\left(x\right)\in\Omega$
for every $x\in\partial\Omega$. Otherwise, there would exist $z\in\partial\Omega$
such that $d\left(x+\frac{1}{2}\varepsilon{n}\left(x\right),z\right)<\frac{1}{2}\varepsilon$,
again leading to a contradiction. Let now $B\left(x+\frac{1}{2}\varepsilon{n}\left(x\right),\frac{1}{2}\varepsilon\right)$
denote the ball radius $\frac{1}{2}\varepsilon$ centered at $x+\frac{1}{2}{n}\left(x\right)$.
We claim that this ball is contained in $\overline{\Omega}$. To see
this, assume there exists $y\in B\left(x+\frac{1}{2}\varepsilon{n}\left(x\right),\frac{1}{2}\varepsilon\right)$
such that $y\notin\overline{\Omega}$. In this case, we can find $z\in\partial\Omega$
satisfying $d\left(x+\frac{1}{2}\varepsilon{n}\left(x\right),z\right)\leqslant\frac{1}{2}\varepsilon$
which as before would yield a contradiction, since $z\neq x$.

Finally let $\rho>2\varepsilon^{-1}$, then $A_{\Omega,\Theta}\left(b-\rho^{-1}\right)$
and $A_{\Omega,\Theta}\left(-a+\rho^{-1}\right)$, that is the parallel section
functions of $\Omega$ are larger than the parallel section functions
of a suitable sphere of radius $\frac{1}{2}\varepsilon$. The lemma
follows.
\end{proof}
\begin{lemma}
\label{lem:concave}Let $f$ be nonnegative and concave on its support
$\left[-a,b\right]$, where $a,b>0$. Also assume $f\left(-a\right)=f\left(b\right)=0$.
Then, if $0\leqslant\lambda_{1}<\lambda_{2}\leqslant b$ or $-a\leqslant\lambda_{2}<\lambda_{1}\leqslant0$,
\begin{equation}
f\left(\lambda_{2}\right)\leqslant\left(1+\max\left(\frac{b}{a},\frac{a}{b}\right)\right)f\left(\lambda_{1}\right),\label{lambda2lambda1}
\end{equation}
In particular for every $x\in\mathbb{R}$
\begin{equation}
f\left(x\right)\leqslant\left(1+\max\left(\frac{b}{a},\frac{a}{b}\right)\right)f\left(0\right)\text{.}\label{2f(0)}
\end{equation}
Moreover, for $0\leqslant\lambda_{1}<\lambda_{2}<b$
\begin{equation}
f\left(\lambda_{1}\right)\leqslant\frac{b-\lambda_{1}}{b-\lambda_{2}}f\left(\lambda_{2}\right)\label{lambda1<lambda2-1}
\end{equation}
and for $-a<\lambda_{2}<\lambda_{1}\leqslant0$
\begin{equation}
f\left(\lambda_{1}\right)\leqslant\frac{a+\lambda_{1}}{a+\lambda_{2}}f\left(\lambda_{2}\right).\label{lambda1<lambda2-2}
\end{equation}
\end{lemma}

\begin{proof}
Assume $0\leqslant\lambda_{1}<\lambda_{2}\leqslant b$. Since $f$
is concave in $\left[-a,\lambda_{2}\right]$ we have
\[
f\left(-a\right)+\frac{f\left(\lambda_{2}\right)-f\left(-a\right)}{\lambda_{2}+a}\left(\lambda_{1}+a\right)\leqslant f\left(\lambda_{1}\right).
\]
The assumption $f\left(-a\right)=0$ gives
\[
f\left(\lambda_{2}\right)\left(\lambda_{1}+a\right)\leqslant\left(\lambda_{2}+a\right)f\left(\lambda_{1}\right).
\]
Then, since $\frac{\lambda_{2}+a}{\lambda_{1}+a}\leqslant\frac{b+a}{a}=1+\frac{b}{a}$,
we obtain 
\[
f\left(\lambda_{2}\right)\leqslant\left(1+\frac{b}{a}\right)f\left(\lambda_{1}\right).
\]
The proof when $-a\leqslant\lambda_{2}<\lambda_{1}\leqslant0$ is
similar, with the constant $1+b/a$ replaced by $1+a/b$.

Let $0\leqslant\lambda_{1}<\lambda_{2}<b$, since $f$ is concave
in $\left[\lambda_{1},b\right]$, we obtain
\[
f\left(b\right)+\frac{f\left(\lambda_{1}\right)-f\left(b\right)}{\lambda_{1}-b}\left(\lambda_{2}-b\right)\leqslant f\left(\lambda_{2}\right),
\]
so that
\[
f\left(\lambda_{1}\right)\leqslant\frac{b-\lambda_{1}}{b-\lambda_{2}}f\left(\lambda_{2}\right).
\]
Then we obtain (\ref{lambda1<lambda2-1}). The proof of (\ref{lambda1<lambda2-2})
is similar.
\end{proof}
\begin{proof}[Proof of Theorem \ref{thm:da sotto}]
We can assume that $\Omega$ contains a ball of
radius $r_{1}$ centered at the origin and it is contained in a ball
of radius $r_{2}$ centered at the origin. In particular, for every
direction $\Theta$, the support $\left[-a,b\right]$ of the parallel section function $A_{\Omega,\Theta}$ satisfies $r_{1}<a,b<r_{2}$. 

Now let 
\begin{equation}
\Delta_{h}^{d}f\left(t\right)=\sum_{k=0}^{d}{d \choose k}\left(-1\right)^{d-k}f\left(t+kh\right).\label{Def_diff_finita}
\end{equation}
be the finite difference operator of order $d$ with step $h$ (see
e.g. \cite{DeVoreLorentz}). If $f\in L^{1}\left(\mathbb{R}\right)$, then for every $\xi\in\mathbb{R}$,
\[
\widehat{\Delta_{h}^{d}f}\left(\xi\right)=\left(e^{2\pi ih\xi}-1\right)^{d}\widehat{f}\left(\xi\right).
\]
Let $0<\alpha<\beta$ constants, with $\alpha$ small and $\beta$
large, to be better specified later. Let $\kappa>0$ be the constant in Corollary \ref{BNW1} so that for every 
$\left\vert\xi\right\vert\geqslant\kappa$, (\ref{Chi=A}) yields,
\begin{align}
\left\vert\widehat{A}_{\Omega,\Theta}(\xi) \right\vert  & = \left\vert\widehat{\chi}_{\Omega}(\xi \Theta) \right\vert\nonumber \\ 
 & \leqslant c  \ \left\vert\xi\right\vert^{-1} \left(A_{\Omega , \Theta} (-a+\left\vert\xi\right\vert^{-1})+A_{\Omega , \Theta} (b-\left\vert\xi\right\vert^{-1})\right) \ . \label{BNW2} 
\end{align}
Then the Parseval identity gives
\begin{align*}
\int_{\mathbb{R}}\left|\Delta_{h}^{d}A_{\Omega ,\Theta}\left(t\right)\right|^{2}dt= & \int_{\mathbb{R}}\left|e^{2\pi ih\xi}-1\right|^{2d}\left|\widehat{A_{\Omega ,\Theta}}\left(\xi\right)\right|^{2}d\xi\\
= & \int_{\left|\xi\right|\leqslant\kappa}\left|e^{2\pi ih\xi}-1\right|^{2d}\left|\widehat{A_{\Omega ,\Theta}}\left(\xi\right)\right|^{2}d\xi\\
 & +\int_{\kappa<\left|\xi\right|\leqslant\alpha\rho}\left|e^{2\pi ih\xi}-1\right|^{2d}\left|\widehat{A_{\Omega ,\Theta}}\left(\xi\right)\right|^{2}d\xi\\
 & +\int_{\alpha\rho\leqslant\left|\xi\right|\leqslant\beta\rho}\left|e^{2\pi ih\xi}-1\right|^{2d}\left|\widehat{A_{\Omega ,\Theta}}\left(\xi\right)\right|^{2}d\xi\\
 & +\int_{\left|\xi\right|\geqslant\beta\rho}\left|e^{2\pi ih\xi}-1\right|^{2d}\left|\widehat{A_{\Omega ,\Theta}}\left(\xi\right)\right|^{2}d\xi\\
= & \ \mathcal{A}+\mathcal{B}+\mathcal{C}+\mathcal{D}.
\end{align*}
From now on we let $h=\rho^{-1}$. To estimate $\mathcal{A}$ observe that
\[
\left|\widehat{A_{\Omega ,\Theta}}\left(\xi\right)\right|\leqslant\int_{-a}^{b}A_{\Omega ,\Theta}\left(t\right)dt=\left|\Omega\right|.
\]
Hence
\begin{align*}
\mathcal{A}  \leqslant\int_{\left|\xi\right|\leqslant\kappa}\left|2\pi\rho^{-1}\xi\right|^{2d}\left|\widehat{A_{\Omega ,\Theta}}\left(\xi\right)\right|^{2}d\xi
  \leqslant\left|\Omega\right|^{2}\left|2\pi\rho^{-1}\right|^{2d}\int_{\left|\xi\right|\leqslant\kappa}\left|\xi\right|^{2d}d\xi\leqslant c\rho^{-2d}.
\end{align*}
To estimate $\mathcal{B}$ and $\mathcal{D}$ we use (\ref{BNW2}),
\begin{align*}
\mathcal{B}\leqslant & \int_{\kappa<\left|\xi\right|\leqslant\alpha\rho}\left|2\pi\rho^{-1}\xi\right|^{2d}\left|\widehat{A_{\Omega,\Theta}}\left(\xi\right)\right|^{2}d\xi\\
\leqslant & c\rho^{-2d}\int_{\kappa<\left|\xi\right|\leqslant\alpha\rho}\left|\xi\right|^{2d-2}\left[A_{\Omega,\Theta}\left(-a+\left|\xi\right|^{-1}\right)^{2}+A_{\Omega,\Theta}\left(b-\left|\xi\right|^{-1}\right)^{2}\right]d\xi
\end{align*}
By Theorem \ref{Brunn}, there exists a function $\phi_\Theta$, concave on its support, such that
\[
A_{\Omega,\Theta}\left(t\right)=\phi_{\Theta}^{d-1}\left(t\right).
\]
If
$\alpha<1$ and $\left|\xi\right|\leqslant\alpha\rho$, then $b-\left|\xi\right|^{-1}\leqslant b-\rho^{-1}$.
Then (\ref{lambda1<lambda2-1}) yields
\[
\phi_{\Theta}\left(b-\left|\xi\right|^{-1}\right)\leqslant\frac{\left|\xi\right|^{-1}}{\rho^{-1}}\phi_{\Theta}\left(b-\rho^{-1}\right).
\]
Then, taking this inequality to the power $d-1$, 
\[
A_{\Omega,\Theta}\left(b-\left|\xi\right|^{-1}\right)\leqslant\left(\frac{\rho}{\left|\xi\right|}\right)^{d-1}A_{\Omega,\Theta}\left(b-\rho^{-1}\right).
\]
Hence
\begin{align*}
\mathcal{B}\leqslant & c\rho^{-2d}\int_{\kappa<\left|\xi\right|\leqslant\alpha\rho}\left|\xi\right|^{2d-2}\left[\left(\frac{\rho}{\left|\xi\right|}\right)^{2d-2}A_{\Omega,\Theta}\left(-a+\rho^{-1}\right)^{2}\right.\\
&+\left. \left(\frac{\rho}{\left|\xi\right|}\right)^{2d-2}A_{\Omega,\Theta}\left(b-\rho^{-1}\right)^{2}\right]d\xi\\
= & c\rho^{-2}\int_{\kappa<\left|\xi\right|\leqslant\alpha\rho}\left[A_{\Omega,\Theta}\left(-a+\rho^{-1}\right)^{2}+A_{\Omega,\Theta}\left(b-\rho^{-1}\right)^{2}\right]d\xi\\
\leqslant & c\rho^{-1}\alpha\left[A_{\Omega,\Theta}\left(-a+\rho^{-1}\right)^{2}+A_{\Omega,\Theta}\left(b-\rho^{-1}\right)^{2}\right].
\end{align*}
Let us consider the term $\mathbb{\mathcal{D}}$. Let $\delta=1+\frac{r_{2}}{r_{1}}$. Since $r_{1}<a,b<r_{2}$ we have
\[
\delta\geqslant1+\max\left(\frac{b}{a},\frac{a}{b}\right).
\]
Hence, if $\left|\xi\right|\geqslant\beta\rho$, by (\ref{lambda2lambda1}),
we have
\[
A_{\Omega,\Theta}\left(-a+\left|\xi\right|^{-1}\right)\leqslant\delta^{d-1}A_{\Omega,\Theta}\left(-a+\rho^{-1}\right),
\]
and
\[
A_{\Omega,\Theta}\left(b-\left|\xi\right|^{-1}\right)\leqslant\delta^{d-1}A_{\Omega,\Theta}\left(b-\rho^{-1}\right).
\]
Arguing as above,
\begin{align*}
\mathcal{D}= & \int_{\left|\xi\right|\geqslant\beta\rho}\left|e^{2\pi ih\xi}-1\right|^{2d}\left|\widehat{A_{\Omega,\Theta}}\left(\xi\right)\right|^{2}d\xi\\
\leqslant & c\int_{\left|\xi\right|\geqslant\beta\rho}\left|\xi\right|^{-2}\left[A_{\Omega,\Theta}\left(-a+\left|\xi\right|^{-1}\right)+A_{\Omega,\Theta}\left(b-\left|\xi\right|^{-1}\right)\right]^{2}d\xi\\
\leqslant & c\delta^{2d-2}\left[A_{\Omega,\Theta}\left(-a+\rho^{-1}\right)^{2}+A_{\Omega,\Theta}\left(b-\rho^{-1}\right)^{2}\right]\int_{\left|\xi\right|\geqslant\beta\rho}\left|\xi\right|^{-2}d\xi\\
\leqslant & c\beta^{-1}\rho^{-1}\left[A_{\Omega,\Theta}\left(-a+\rho^{-1}\right)^{2}+A_{\Omega,\Theta}\left(b-\rho^{-1}\right)^{2}\right].
\end{align*}
Finally
\begin{align*}
\mathcal{C}= & \int_{\alpha\rho\leqslant\left|\xi\right|\leqslant\beta\rho}\left|e^{2\pi ih\xi}-1\right|^{2d}\left|\widehat{A_{\Omega,\Theta}}\left(\xi\right)\right|^{2}ds\leqslant & 2^{2d}\int_{\alpha\rho\leqslant\left|\xi\right|\leqslant\beta\rho}\left|\widehat{A_{\Omega,\Theta}}\left(\xi\right)\right|^{2}ds.
\end{align*}
Hence 
\begin{align}
& \int_{\mathbb{R}}\left|\Delta_{h}^{d}A_{\Omega,\Theta}\left(t\right)\right|^{2}dt \nonumber\\
\leqslant & c\rho^{-2d}+c\alpha\rho^{-1}\left[A_{\Omega,\Theta}\left(-a+\rho^{-1}\right)^{2}+A_{\Omega,\Theta}\left(b-\rho^{-1}\right)^{2}\right]\nonumber \\
 & +2^{2d}\int_{\mathbb{\alpha\rho\leqslant\left|\xi\right|\leqslant\beta\rho}}\left|\widehat{A_{\Omega,\Theta}}\left(\xi\right)\right|^{2}d\xi\label{eq:A+B+C}\\
 & +c\beta^{-1}\rho^{-1}\left[A_{\Omega,\Theta}\left(-a+\rho^{-1}\right)^{2}+A_{\Omega,\Theta}\left(b-\rho^{-1}\right)^{2}\right].\nonumber 
\end{align}
Since $A_{\Omega, \Theta}\left(t+kh\right)$ has support in the interval $\left[-a-kh,b-kh\right]$, 
it follows from  (\ref{Def_diff_finita}) that if $b-h<t<b$ then $\Delta_{h}^{n}A_{\Omega,\Theta}\left(t\right)=A_{\Omega,\Theta}\left(t\right)$.
Hence, by (\ref{lambda2lambda1}),
\begin{align*}
\int_{\mathbb{R}}\left|\Delta_{h}^{d}A_{\Omega,\Theta}\left(t\right)\right|^{2}dt\geqslant & \int_{b-h}^{b}\left|A_{\Omega,\Theta}\left(t\right)\right|^{2}dt\geqslant\int_{b-h}^{b-h/2}\left|A_{\Omega,\Theta}\left(t\right)\right|^{2}dt\\
\geqslant\delta^{2d-2} & \frac{h}{2}A_{\Omega,\Theta}\left(b-h/2\right)^{2}\geqslant c\rho^{-1}A_{\Omega,\Theta}\left(b-\rho^{-1}\right)^{2}.
\end{align*}
A similar computation also gives
\[
\int_{\mathbb{R}}\left|\Delta_{h}^{d}A_{\Omega,\Theta}\left(t\right)\right|^{2}dt\geqslant c\rho^{-1}A_{\Omega,\Theta}\left(-a+\rho^{-1}\right)^{2}.
\]
Therefore
\[
\int_{\mathbb{R}}\left|\Delta_{h}^{d}A_{\Omega,\Theta}\left(t\right)\right|^{2}dt\geqslant c\rho^{-1}\left[A_{\Omega,\Theta}\left(-a+\rho^{-1}\right)^{2}+A_{\Omega,\Theta}\left(b+\rho^{-1}\right)^{2}\right].
\]
Note that this last constant $c>0$ only depends on $\delta=1+r_2/r_1$ hence it is independent of $\Theta$. Using
(\ref{eq:A+B+C}) we obtain
\begin{align}
 & c\rho^{-1}\left[A_{\Omega,\Theta}\left(-a+\rho^{-1}\right)^{2}+A_{\Omega,\Theta}\left(b+\rho^{-1}\right)^{2}\right]\nonumber \\
\leqslant & c\rho^{-2d}+c\alpha\rho^{-1}\left[A_{\Omega,\Theta}\left(-a+\rho^{-1}\right)^{2}+A_{\Omega,\Theta}\left(b-\rho^{-1}\right)^{2}\right]\nonumber\\
 & +2^{2d}\int_{\mathbb{\alpha\rho\leqslant\left|\xi\right|\leqslant\beta\rho}}\left|\widehat{A_{\Omega,\Theta}}\left(\xi\right)\right|^{2}d\xi\label{eq:A+B+C-1}\\
 & +c\beta^{-1}\rho^{-1}\left[A_{\Omega,\Theta}\left(-a+\rho^{-1}\right)^{2}+A_{\Omega,\Theta}\left(b-\rho^{-1}\right)^{2}\right].\nonumber 
\end{align}
Hence (\ref{Chi=A}) gives
\begin{align*}
&\int_{\alpha\rho\leqslant s\leqslant\beta\rho} \left|\widehat{\chi_{C}}\left(s\Theta\right)\right|^{2}ds \\
\geqslant & c\left(1-\alpha-\beta^{-1}\right)\rho^{-1}\left[A_{\Omega,\Theta}\left(-a+\rho^{-1}\right)^{2}+A_{\Omega,\Theta}\left(b-\rho^{-1}\right)^{2}\right]-c\rho^{-2d}.
\end{align*}
If we take $\alpha$ suitably small and $\beta$ suitably large we
obtain
\begin{align*}
& \int_{\alpha\rho\leqslant s\leqslant\beta\rho}\left|\widehat{\chi_{C}}\left(s\Theta\right)\right|^{2}ds \\
\geqslant & c\rho^{-1}\left[A_{\Omega,\Theta}\left(-a+\rho^{-1}\right)^{2}+A_{\Omega,\Theta}\left(b-\rho^{-1}\right)^{2}\right]-c\rho^{-2d}.
\end{align*}
Since by Lemma \ref{lem:uniform ball condition}, 
\[
c\rho^{-1}\left[A_{\Theta}\left(-a+\rho^{-1}\right)^{2}+A_{\Theta}\left(b-\rho^{-1}\right)^{2}\right]\geqslant c\rho^{-d}
\]
if $\rho$ is sufficiently large we obtain the required estimate.
\end{proof}
\begin{corollary}
\label{cor:Fourier da sotto}Let $\Omega$ be a bounded convex set
in $\mathbb{R}^{d}$. Assume that $\partial\Omega$ is smooth with
finite order of contact at every point. Then there exist positive
constants $\gamma,c,C$ such that for every direction $\Theta$ and
$\rho\geqslant c,$ 
\begin{equation}
\int_{\gamma\rho\leqslant s\leqslant\rho}\left|\widehat{\chi_{\Omega}}\left(s\Theta\right)\right|^{2}ds\geqslant C\rho^{-d}.\label{eq:stima da sotto}
\end{equation}
\end{corollary}

\begin{proof}
By Thorem \ref{thm:da sotto} and Lemma \ref{lem:uniform ball condition}
we have
\begin{align*}
\int_{\alpha\rho\leqslant s\leqslant\beta\rho}\left|\widehat{\chi_{C}}\left(s\Theta\right)\right|^{2}ds\geqslant & c\rho^{-1}\left[A_{\Theta}\left(-a+\rho^{-1}\right)^{2}+A_{\Theta}\left(b-\rho^{-1}\right)^{2}\right]\\
\geqslant & c\rho^{-1}\rho^{-\left(d-1\right)}=c\rho^{-d}.
\end{align*}
Let $\rho=\rho'/\beta$, then
\[
\int_{\frac{\alpha}{\beta}\rho'\leqslant s\leqslant\rho'}\left|\widehat{\chi_{C}}\left(s\Theta\right)\right|^{2}ds\geqslant c\beta^{d}\left(\rho'\right)^{-d}.
\]
\end{proof}

\begin{remark}
Since
\begin{align*}
\int_{\Sigma_{d-1}}\int_{\alpha\rho\leqslant s\leqslant\beta\rho}\left|\widehat{\chi_{\Omega}}\left(s\Theta\right)\right|^{2}dsd\Theta & \approx\rho^{1-d}\int_{\Sigma_{d-1}}\int_{\alpha\rho\leqslant s\leqslant\beta\rho}\left|\widehat{\chi_{\Omega}}\left(s\Theta\right)\right|^{2}s^{d-1}dsd\Theta\\
 & =\rho^{1-d}\int_{\left\{ \alpha\rho\leqslant\left|\xi\right|\leqslant\beta\rho\right\} }\left|\widehat{\chi_{\Omega}}\left(\xi\right)\right|^{2}d\xi
\end{align*}
from Lemma 4.2 in \cite{BCT} we readily obtain the following average estimate
\[
\int_{\Sigma_{d-1}}\int_{\alpha\rho\leqslant s\leqslant\beta\rho}\left|\widehat{\chi_{\Omega}}\left(s\Theta\right)\right|^{2}dsd\Theta\approx\rho^{-d}.
\]
which shows that the RHS $\rho^{-d}$ in (\ref{eq:stima da sotto})
cannot be replaced with $\rho^{-d+\varepsilon}$. 
\end{remark}

\section{Proof of Theorem \ref{d3}}
For the proof of the theorem we need an estimate due to J.W. Cassels and H. Montgomery (see \cite[Ch. 5, Theorem 12]{montgomery94}), which we reformulate as follows.
\begin{lemma}
For every $H>0$ there exist positive constants $c$ and $C$ such that for every choice of $N$ points 
$z_1,\ldots,z_N$ in $\mathbb{T}^d$, if $M\geqslant C\,N^{1/d}$ then
$$
\sum_{H<\left|n\right|<M}\left|\sum_{j=1}^{N}e^{2\pi in\cdot z_{j}}\right|^{2}\geqslant c NM^d.
$$
\end{lemma}
\begin{proof}
The estimate with $H=0$ is the $d$-dimensional analog of Theorem 12 in Chaper 5 of \cite{montgomery94}. If  $H>0$ and $M>C\,N^{1/d}$, with $C$ suitably large, we have
\begin{align*}
\sum_{H<\left|n\right|<M}\left|\sum_{j=1}^{N}e^{2\pi in\cdot z_{j}}\right|^{2} & =\sum_{0<\left|n\right|<M}\left|\sum_{j=1}^{N}e^{2\pi in\cdot z_{j}}\right|^{2}-\sum_{0<\left|n\right|\leqslant H}\left|\sum_{j=1}^{N}e^{2\pi in\cdot z_{j}}\right|^{2}\\
 & \geqslant cNM^{d}-H^{d}N^{2}\geqslant cNM^{d}.
\end{align*}
\end{proof}
\begin{proof}[Proof of Theorem  \ref{d3}]
The discrepancy function 
\begin{align*}
D_{N}\left(x\right)= & \sum_{j=1}^{N}\chi_{x+r\Omega}\left(z_{j}\right)-N\left|r\Omega\right|
\end{align*}
is a periodic function of $x$ with Fourier expansion
\[
D_{N}\left(x\right)=\sum_{n\neq0}\widehat{D_{N}}\left(n\right)e^{2\pi in\cdot x}
\]
where
\[
\widehat{D_{N}}\left(n\right)=\left(\sum_{j=1}^{N}e^{2\pi in\cdot z_{j}}\right)r^{d}\widehat{\chi_{\Omega}}\left(rn\right).
\]
Hence by Parseval theorem, for every $0<H<M$ one has
\begin{align*}
 & \int_{\mathbb{T}^{d}}\int_{0}^{1}\left|\sum_{j=1}^{N}\chi_{x+r\Omega}\left(z_{j}\right)-N\left|r\Omega\right|\right|^{2}dxdr\\
= & \sum_{n\neq0}\left|\sum_{j=1}^{N}e^{2\pi in\cdot z_{j}}\right|^{2}\int_{0}^{1}\left|r^{d}\widehat{\chi_{\Omega}}\left(rn\right)\right|^{2}dr\\
\geqslant & \sum_{H<\left|n\right|<M}\left|\sum_{j=1}^{N}e^{2\pi in\cdot z_{j}}\right|^{2}\int_{0}^{1}\left|r^{d}\widehat{\chi_{\Omega}}\left(rn\right)\right|^{2}dr\\
\geqslant & \left[\inf_{H<\left|n\right|<M}\int_{0}^{1}\left|r^{d}\widehat{\chi_{\Omega}}\left(rn\right)\right|^{2}dr\right]\left[\sum_{H<\left|n\right|<M}\left|\sum_{j=1}^{N}e^{2\pi in\cdot z_{j}}\right|^{2}\right] = \mathcal{A}\cdot\mathcal{B}.
\end{align*}
To estimate $\mathcal{A}$ let $\Theta=|n|^{-1}n$ and observe that by Corollary \ref{cor:Fourier da sotto}, if $\left|n\right|\geqslant H$,
with $H$ suitably large, we have
\begin{align*}
\int_{0}^{1}\left|r^{d}\widehat{\chi_{\Omega}}\left(rn\right)\right|^{2}dr= & \left|n\right|^{-2d-1}\int_{0}^{\left|n\right|}s^{2d}\left|\widehat{\chi_{\Omega}}\left(s\Theta\right)\right|^{2}ds\\
\geqslant & \left|n\right|^{-2d-1}\int_{\varepsilon\left|n\right|}^{\left|n\right|}s^{2d}\left|\widehat{\chi_{\Omega}}\left(s\Theta\right)\right|^{2}ds\\
\geqslant & \varepsilon^{2d}\left|n\right|^{-1}\int_{\varepsilon\left|n\right|}^{\left|n\right|}\left|\widehat{\chi_{C}}\left(s\Theta\right)\right|^{2}ds\geqslant c\left|n\right|^{-d-1}.\\
\end{align*}
Hence
\[
\left[\inf_{H<\left|n\right|<M}\int_{0}^{1}\left|r^{d}\widehat{\chi_{C}}\left(rn\right)\right|^{2}dr\right]\geqslant cM^{-d-1}.
\]
For the term $\mathcal{B}$, if $M=CN^{1/d}$ the above lemma gives 
\begin{align*}
\sum_{H<\left|n\right|<M}\left|\sum_{j=1}^{N}e^{2\pi in\cdot z_{j}}\right|^{2}  \geqslant cNM^{d}.
\end{align*}
Finally, we conclude that
\[
\int_{\mathbb{T}^{d}}\int_{0}^{1}\left|\sum_{j=1}^{N}\chi_{x+rC}\left(z_{j}\right)-N\left|rC\right|\right|^{2}dxdr\geqslant cM^{-d-1}\,NM^{d}=cN^{1-1/d}.
\]
\end{proof}

\end{document}